\documentclass[oneside,reqno]{amsart}
\usepackage{a4}
\usepackage{amssymb}
\usepackage{mathrsfs}

\newtheorem{Th}{Theorem}[section]
\newtheorem{Cor}[Th]{Corollary}
\newtheorem{Lem}[Th]{Lemma}
\newtheorem{Prop}[Th]{Proposition}

\newtheorem{claim-num}{Claim}

\newtheorem*{theo}{Theorem}
\newtheorem*{defn}{Definition}
\newtheorem*{rem}{Remark}

\def\a{\alpha}
\def\eps{\varepsilon}
\def\vk{\varkappa}
\def\f{\varphi}
\def\s{\sigma}

\def\N{\mathbf N}

\def\aut#1{\mathrm{Aut}(#1)}
\def\str#1{\langle#1\rangle}

\def\sym#1{\mathrm{Sym}(#1)}

\def\inv{{}^{-1}}
\def\sle{\subseteq}
\def\Mod#1{\ (\mathrm{mod}\ #1)}

\def\id{\mathrm{id}}

\renewcommand{\le}{\leqslant}
\renewcommand{\ge}{\geqslant}

\def\To{\Rightarrow}

\def\rank{\operatorname{rank} }

\def\avst#1{ \overline{\mathstrut #1} }

\def\symf#1{\mathrm{Sym}_F(#1)}

\def\mobf{\mathrm{Mo}_\cB(F)}
\def\vk{\varkappa}

\def\cM{\mathcal M}

\def\cB{\mathscr B}
\def\cC{\mathscr C}
\def\cD{\mathscr D}
\def\cE{\mathscr E}
\def\cF{\mathscr F}
\def\cG{\mathscr G}

\numberwithin{equation}{section}

\begin{document}

\begin{abstract}
Let $F$ be a relatively free algebra of infinite rank
$\vk.$ We say that $F$ has the {\it small index property}
if any subgroup of $\Gamma=\aut F$ of index at most $\vk$ contains
the pointwise stabilizer $\Gamma_{(U)}$ of a subset
$U$ of $F$ of cardinality less than $\vk.$ We prove that
every infinitely generated free nilpotent/abelian
group has the small index property, and discuss
a number of applications.
\end{abstract}

\title[Free nilpotent groups]
{The small index property for free nilpotent groups}
\author{Vladimir Tolstykh}
\address{Vladimir Tolstykh\\ Department of Mathematics\\ Yeditepe University\\
34755 Kay\i\c sda\u g\i \\
Istanbul\\
Turkey}
\email{vtolstykh@yeditepe.edu.tr}
\maketitle


\section{Introduction}

A countable first-order structure $\cM$ is said to have the {\it small index
property} if every subgroup of the automorphism group $\Gamma=\aut\cM$
of index less than $2^{\aleph_0}$ contains the pointwise
stabilizer $\Gamma_{(U)}$ of a finite subset $U$ of the
domain of $\cM$ \cite[Section 4.2]{Hodges}.
Here we use and shall use in the rest of the paper
the standard notation of the theory of permutation groups. Working with a group $G$
which acts on set $X,$ we shall denote by $G_{(Y)}$
and by $G_{\{Y\}}$ the pointwise and the setwise
stabilizer of a subset $Y$ of $X$ in $G,$ respectively. Any symbol
of the form $G_{*_1,*_2}$ will denote the intersection
of stabilizers $G_{*_1}, G_{*_2} \le G$ (e.g.
$G_{(Y),\{Z\}}=G_{(Y)} \cap G_{\{Z\}}).$

There are many
examples of countable structures that have the small index property:
any countable set with no relations; the set of rational
numbers with the natural order; any countable
atomless boolean algebra, any vector space of countable
dimension over an at most countable field, etc. The special
role played by the cardinal $\aleph_0$ is (essentially)
explained by the fact that if $\cM$ is a first-order structure
of cardinality $\aleph_0,$ then $\aut \cM$ can be naturally
converted into a Polish topological group \cite[pp. 59--60]{Kech} (recall
that a topological space is Polish if and only
if it is separable and completely metrizable).

One of the attractions of the small index property
lies in the fact that if for a countable structure
$\cM$ we have that $|\aut \cM|=2^{\aleph_0},$
then any automorphism $\Delta \in \aut \Gamma$ of the group $\Gamma=\aut \cM$ takes
a subgroup of small ($< 2^{\aleph_0})$ index to
a subgroup of small index, since the condition
``$\Sigma$ is of small index in $\Gamma$'' and
``$|\Gamma : \Sigma| < |\Gamma|$'' are in this case equivalent.
This makes the small index property an efficient
tool to prove results on the isomorphism types
of the automorphism groups of countable structures,
to prove results on reconstruction of countable structures from
their automorphism groups, etc. These considerations lead
to a search of a reasonable analogue of the small index
property for arbitrary infinite structures. One of the
crucial results in this direction is the following theorem
by J. Dixon, P.M. Neumann and S. Thomas which appeared
in one of the very first papers on the small index property \cite{DiNeuTho}.

\begin{theo}[Theorem $2^{\mbox{\normalsize$\flat$}}$ of \cite{DiNeuTho}]
Let $I$ be an infinite set. Then any subgroup of the
symmetric group $\Gamma=\sym I$ of $I$ of index at most $|I|$ contains
the pointwise stabilizer $\Gamma_{(U)}$ a subset $U$ of $I$ of cardinality
$< |I|.$
\end{theo}

Note that if $I$ is of cardinality $\aleph_0,$ then for any
subgroup $\Sigma$ of $\Gamma=\sym I$ the conditions
``$|\Gamma:\Sigma| \le \aleph_0$'' and ``$|\Gamma:\Sigma| < 2^{\aleph_0}$''
are equivalent \cite[Th. 4.2.8]{Hodges}.

Working with relatively free algebras $F,$ one quite often
meets the situation when the automorphism group $\aut F$
``recognizes'' the rank of $F$ (the cardinality of a basis
of $F$). For instance (keeping in mind Theorem $2^{\mbox{\normalsize$\flat$}}$ of \cite{DiNeuTho}),
if $F$ is an infinite set with no relations (a relatively
free algebra in the empty language), then $\rank(F)=|F|$
is equal to the cardinality of the family of all transpositions
in $\aut F$ (the family of all transpositions can be
characterized in $\aut F$ in terms of group operation,
see e.g. \cite{McK}). Similarly, $\aut F$ ``recognizes''
$\rank(F)$ if $F$ is an infinite-dimensional vector space over a division
ring, an infinitely generated free nilpotent/abelian, or an infinitely
generated centerless relatively free group \cite[Prop. 1.3]{To_Smallness},
etc. Summarizing the above discussion, we suggest
the following definition.

\begin{defn}
Let $F$ be a relatively free algebra of infinite rank $\vk.$
We shall say that $F$ has the {\it small index property}
if any subgroup $\Sigma$ of $\Gamma=\aut F$ of index at most $\vk$ contains
the pointwise stabilizer $\Gamma_{(U)}$ of a subset
$U$ of $F$ of cardinality $< \vk.$
\end{defn}

Henceforth we shall use the term the small index
property only in the sense of the above definition
(which a priori differs from the classical
one for countable structures).

The paper is organized as follows. Proposition \ref{MStab-in-a-small-index-subgroup}
of Section 2 states that given a relatively free
algebra $F$ of infinite rank with a basis $\cB$
and any subgroup $\Sigma \le \Gamma=\aut F$ of small
$(\le \rank(F))$ index, there is a moiety $\cC$ of
$\cB$ such that $\Sigma$ contains the subgroup
$\Gamma_{(\cB \setminus \cC),\{\str{\cC}\}}$ of
all elements of $\Gamma$ which fix $\cB \setminus \cC$
pointwise and fix the subalgebra $\str{\cC}$
generated by $\cC$ setwise. The proof of Proposition
\ref{MStab-in-a-small-index-subgroup} relies heavily
on Theorem $2^{\mbox{\normalsize$\flat$}}$ of \cite{DiNeuTho}
we have quoted above.

Then we consider a number of corollaries of Proposition \ref{MStab-in-a-small-index-subgroup}
for free algebras of infinite rank in so-called BMN-varieties
introduced by the author in \cite{To_Berg}. A variety $\mathfrak V$
of algebras is called a {\it BMN-variety}, if given a free
algebra $F \in \mathfrak V$ of infinite rank and any
basis $\cB$ of $F,$ the group $\Gamma=\aut F$ is generated
by the stabilizers
$$
\Gamma_{(\cB_1),\{\str{\cB_2 \cup \cC}\}} \text{ and }
\Gamma_{(\cB_2),\{\str{\cB_1 \cup \cC}\}}
$$
where $\cB=\cB_1 \sqcup \cB_2 \sqcup \cC$ is any
partition of $\cB$ into moieties. Examples of BMN-varieties
are, for instance, the variety of algebras with no structure,
any variety of vector spaces over a fixed
division ring, any variety $\mathfrak N_c$ of nilpotent groups of
class $\le c,$ etc. By applying Proposition \ref{MStab-in-a-small-index-subgroup},
we prove that if $F$ is a free algebra of infinite
rank from a BMN-variety, then the group $\Gamma=\aut F$ has no
proper normal subgroups of index $< 2^{\rank(F)}$ and that
every subgroup $\Sigma$ of $\Gamma$ of small index contains
a stabilizer $\Gamma_{(\cD),\{\str{\cB \setminus \cD}\}}$
where $\cB$ is a basis of $F$ and $\cD$ is a subset
of $\cB$ of cardinality $< \rank(F).$ To compare
these results with the situation in the general
case, we show, using the ideas from the paper
\cite{BrRom} by R.~Bryant and V.~Roman'kov,
that if $F$ is a relatively free algebra of
infinite rank with the small index property,
then any proper normal subgroup of $\aut F$ is of index
$> \rank(F)$ and that the group $\aut F$ is perfect
(Proposition \ref{SmInd=>Perf}).

Then we consider a natural sufficient condition for any infinitely
generated free algebra from an BMN-variety to have the
small index property (Proposition \ref{Stab=<Shariki>})
and show that this condition is true for any free nilpotent/abelian
group of infinite
rank (Theorem \ref{SmIndProp4Nilps}); the said condition
is also true, for instance, for any infinite-dimensional
vector space over a division ring. Note that R.~Bryant and D.~Evans proved in \cite{BrEv} that
any free group of countably infinite rank has the small index
property and obtained as a corollary that some other
relatively free groups of countably infinite rank (in
particular, free nilpotent groups of
countably infinite rank) have the small index property.
Thus the result by R.~Bryant and D.~Evans concerning
free nilpotent groups of countably infinite
rank is transferred to free nilpotent groups
of arbitrary infinite rank.

\section{The small index property for relatively free algebras}

Let $F$ be a relatively free algebra with a basis $\cB.$
We call an automorphism $\pi$ of $F$ a {\it permutational}
automorphism with respect to $\cB,$ if $\pi$ preserves
$\cB$ as a set; the group of all $\cB$-permutational
automorphisms of $F$ will be denoted by $\symf\cB;$
clearly, $\symf\cB \cong \sym\cB.$

Recall that a subset $J$ of an infinite set $I$
is called a {\it moiety} of $I$ if $|J|=|I \setminus J|.$
Given an automorphism $\s$ of $F$ in the case
when $F$ is of infinite rank, we shall call $\s$ {\it moietous}
with regard to $\cB,$ if there is a partition
$\cB=\cB_1 \sqcup \cB_2$ of $\cB$ into moieties such that $\s \in \Gamma_{(\cB_1),\{\str{\cB_2}\}},$
that is, if $\s$ fixes $\cB_1$ pointwise and fixes the subalgebra
$\str{\cB_2}$ generated by $\cB_2$ setwise.
The set of all $\cB$-moietous automorphisms
of $F$ will be denoted by $\mathrm{Mo}_\cB(F).$

Now let
$$
\cB = \bigsqcup_{i \in I} \cB_i
$$
be a partition of $\cB.$ Having a situation like that we shall write
\begin{equation}
F = \circledast_{i \in I} F_i
\end{equation}
where $F_i =\str{\cB_i}$ is the subalgebra
of $F$ generated by $\cB_i$ $(i \in I);$
thus $F$ is the coproduct of subalgebras $F_i.$ Now given
automorphisms $\f_i \in \aut{F_i}$ where $i$ runs over $I,$
there is a uniquely determined automorphism
$\f \in \aut F$ such that $\f|_{F_i}=\f_i$
for all $i \in I;$ in the manner of (\theequation) we write
$$
\f = \circledast_{i \in I} \f_i.
$$
Suppose that for a certain subset $J \sle I$
and for a certain $j_0 \in J$ we have that
\begin{itemize}
\item $|\cB_j|=|\cB_{j_0}|$ for all $j \in J$ (so all
sets $\cB_j$ are equipotent);
\item for every $j \in J$ there is a bijection $p_j : \cB_{j_0} \to \cB_j$
such that $\f_j =  \pi_j \f_{j_0} \pi_j\inv$
where $\pi_j$ is the isomorphism of algebras in
the language of $F$ induced by $p_j$ (and then the actions of all $\f_j$ are isomorphic);
\end{itemize}
it is then convenient to write the automorphism
$\f = \circledast_{i \in I} \f_i$
in the following simplified form
$$
\f = (\circledast_{j \in J} \f_{j_0}) \circledast (\circledast_{i \in (I \setminus J)} \f_i).
$$
Further, let
$$
\cB =\cD \sqcup \bigsqcup_{j \in J} \cC_j \sqcup \bigsqcup_{i \in I} \cC_i.
$$
be a partition of $\cB.$ Suppose that for all $j \in J,$
$\rho_j$ is an automorphism of the subalgebra
$\str{\cD,\cC_j}$ which fixes $\cD$ pointwise and
that for all $i \in I,$ $\f_i$ is an automorphism
of the subalgebra $\str{\cC_i}.$ Then there exists a uniquely
determined automorphism $\s$ of $F$ such that
\begin{itemize}
\item $\s|_\cD = \id_\cD;$
\item $\s|_{\cC_j} = \rho_j|_{\cC_j}$ for all $j \in J;$
\item $\s|_{\cC_i} = \f_i|_{\cC_i}$ for all $i \in I$;
\end{itemize}
in the manner we have used above, $\s$ can be written as
$$
\s =\id \circledast (\circledast_{j \in J} \rho_j) \circledast
(\circledast_{i \in I} \f_i).
$$

\begin{Prop} \label{MStab-in-a-small-index-subgroup}
Let $F$ be a relatively free algebra of infinite
rank $\vk$ with a basis $\cB.$ Then for every
subgroup $\Sigma$ of $\Gamma=\aut F$ of index at most $\vk$ there is a subset
$\cD$ of $\cB$ of cardinality $< \vk$ such that
$$
\mathrm{Mo}_\cB(F)_{(\cD)} \sle \Sigma
$$
where $\mathrm{Mo}_\cB(F)_{(\cD)}$
is the set of all $\cB$-moietous automorphisms which fix
$\cD$ pointwise. In particular, if $\cC$
is any moiety of $\cB \setminus \cD,$ the stabilizer
$\Gamma_{(\cB \setminus \cC),\{ \str{\cC} \}}$
is contained in $\Sigma.$
\end{Prop}

\begin{proof} The result is proved in \cite{To_Smallness}
for infinitely generated relatively free groups, and in fact
the proof given in \cite{To_Smallness} can be used
without significant changes in the general case.
For the reader's convenience, we reproduce
the plan of the proof.

(a). The key fact is that there exists a subset $\cD$
of $\cB$ of cardinality $< \vk$ such that
the stabilizer
$$
\symf\cB{}_{(\cD)}
$$
is contained in $\Sigma.$ This follows from
Theorem $2^{\mbox{\normalsize$\flat$}}$ of \cite{DiNeuTho}
quoted in the Introduction.

(b). Now let $\cC$ be a moiety of $\cB \setminus \cD.$ Consider
a partition
\begin{equation} \label{Part0cBmcD}
\cB \setminus \cD = \bigsqcup_{i \in I} \cC_{i}
\end{equation}
of $\cB \setminus \cD$ into moieties such that $I$
is an index set of cardinality $\vk$ and $\cC$ is a moiety
of $\cC_{i_0}$ for some $i_0 \in I.$
Then partition $\cC_{i_0}$ into $\aleph_0$ moieties:
\begin{equation}
\cC_{i_0}=\bigsqcup_{k \in \N} \cC_{i_0,k}
\end{equation}
so that $\cC_{i_0,0}=\cC.$ Take any automorphism $\a$ of the subalgebra
$\str{\cD,\cC_{i_0,0}}=\str{\cD,\cC}$
generated by $\cC \cup \cD$ which fixes $\cD$ pointwise.
Extend $\a$ on $\cC_{i_0}$ as follows:
$$
\beta = (\a \circledast \a\inv \circledast \id) \circledast (\a \circledast \a\inv \circledast \id) \circledast \ldots
$$
where the automorphism in the right-hand side corresponds to
the partition (\theequation).

The family $\Lambda=\{\lambda\}$ of automorphisms of $F$ such that
$$
\lambda=\id \circledast (\circledast_{i \in I} \beta^{\varepsilon_i})
$$
where an automorphism in the right-hand side corresponds
to the partition
$$
\cB = \cD \sqcup \bigsqcup_{i \in I} \cC_i
$$
and where $\varepsilon_i=0,1$ ($i\in I$), has the cardinality $2^\vk.$
Then there are distinct $\lambda_1,\lambda_2 \in \Lambda$
such that $\mu=\lambda_1 \lambda_2\inv$ is contained
in $\Sigma$ (in
fact, {\it in any subgroup of $\Gamma$ of index $< 2^\vk;$}
we shall use this fact below).

(c). Now one can find $\cB$-permutational automorphisms
$\pi_1,\pi_2 \in \symf\cB{}_{(\cD)}$ such that the product
$(\mu^{\pi_1} \mu)^{\pi_2}$ acts as $\a$ on $\cC=\cC_{i_0,0}$
and fixes the rest $\cB \setminus \cC$ of $\cB$
pointwise.
\end{proof}

Next, we are going to discuss some corollaries of Proposition
\ref{MStab-in-a-small-index-subgroup}.

A variety $\frak V$ of algebras is called
a {\it BMN-variety} \cite{To_Berg} if given
any free algebra $F$ of $\frak V$ of infinite
rank, any basis $\cB$ of $F$ and any partition
$$
\cB = \cB_1 \sqcup \cB_2 \sqcup \cC
$$
of $\cB$ into moieties, the automorphism group $\Gamma=\aut F$
of $F$ is generated by the stabilizers
$$
\Gamma_{(\cB_1),\{\str{\cB_2 \cup \cC}\}} \text{ and }
\Gamma_{(\cB_1),\{\str{\cB_2 \cup \cC}\}}.
$$
Equivalently, $\Gamma$ is generated:
(a) by any stabilizer of the form
\begin{equation}
\Gamma_{(\cC),\{\str{\cB \setminus \cC}\}}
\end{equation}
where $\cC$ is any moiety of $\cB$ and by all
$\cB$-permutational automorphisms; (b) by any
stabilizer of the form (\theequation) and by a suitable
$\cB$-permutational automorphism; (c) by
all $\cB$-moietous automorphisms of $F.$ One
of the nice properties of the group $\Gamma=\aut F$ is
that $\Gamma$ is a perfect group \cite[Th. 1.5]{To_Berg},
that is, $\Gamma$ is equal to its commutator subgroup
$[\Gamma,\Gamma].$

For instance, the variety
of all sets with no structure (as a corollary
of the results in \cite{DiNeuTho,MN}),
any variety of all vector
spaces over a fixed division ring (as a corollary
of the results in \cite{Macph}),
the variety of all abelian groups
and any variety $\frak N_c$ of nilpotent
groups of class $\le c$ are BMN-varieties
\cite{To_Berg}.

\begin{Prop} \label{BMN=>Harami}
Let $\frak V$ be a BMN-variety of algebras,
$F$ a free algebra of $\frak V$ of infinite
rank $\vk,$ $\cB$ a basis of $F$ and $\Gamma=\aut F.$ Then

{\em (i)} every proper normal subgroup
of $\Gamma$ has index $2^{\vk};$

{\rm (ii)} if $\nu > 1$ is any cardinal
such that $2^\nu < 2^\vk$ {\rm(}for instance,
any cardinal $\nu > 1$ such that $2^\nu \le \vk${\rm)} then $\Gamma$ has no
subgroups of index $\nu;$

{\em (iii)} if $\Sigma$ is a subgroup of $\Gamma$
of index $\le \vk,$ then there is
a subset $\cD$ of $\cB$ of cardinality $< \vk$ such that
the stabilizer $\Gamma_{(\cD),\{\str{\cB \setminus \cD}\}}$
is contained in $\Sigma.$
\end{Prop}

\begin{proof}
(i) It follows from the proof of the part (i) of
Theorem 2.5 in \cite{To_Berg} that $\Gamma$
is the normal closure of a suitable permutational
automorphism. Namely, if
$$
\cB = \cB_1 \sqcup \cB_2 \sqcup \cC
$$
is a partition of $\cB$ into moieties,
then for any $\cB$-permutational automorphism $\pi^*$
of order two of $F$ which
interchanges $\cB_1$ and $\cB_2$ and fixes
$\cC$ pointwise we have that $\Gamma$
is the normal closure of $\pi^*.$ Observe
that the support of the restriction $\pi^*|_\cB$
of $\pi^*$ on $\cB$ is a moiety of $\cB.$

Let $N$ be a normal subgroup of $\Gamma$ having index less than $2^{\vk}.$
Take a moiety $\cC$ of $\cB$ and consider a permutation
$r$ of order two of $\cC$ such that the support
of $r$ is a moiety of $\cC.$ Let $\rho$ be the
automorphism of the subalgebra $\str{\cC}$ induced
by $r.$ Apply then for $\cB,\cC$ and $\rho$ the argument
of part (b) of the proof of Proposition \ref{MStab-in-a-small-index-subgroup}
(assuming that $\cD=\varnothing$). It then follows that
$N$ contains a nonidentity automorphism of the form
$
\mu=\circledast_{i \in I} \beta(\rho)^{\nu_i}
$
where $\nu_i=-1,0,1\quad [=0,1]$ $(i \in I).$ It is easy to
see that the restriction of $\mu$ on $\cB$ is an
involution whose support
is a moiety of $\cB.$
Thus $N$ contains a conjugate of $\pi^*$ described
above, and hence $N=\Gamma,$ as claimed.

(ii) If $H$ is a subgroup of a given group $G,$
then the kernel of the natural action of $G$
on the quotient set $G/H$ is of index $\le |\text{Sym}(G/H)|.$
Thus if $\Gamma$ has a subgroup of index $\nu,$
it has a normal subgroup of index at most $2^\nu,$ contradicting
(i).

(iii) We continue to use the notation introduced
in (i). Let $\Sigma$ be a subgroup of $\Gamma$
of index at most $\vk.$ Then there is a subset $\cD$ of $\cB$ of cardinality
$< \vk$ such that $\Sigma \ge \symf\cB{}_{(\cD)},$
and, by Proposition \ref{MStab-in-a-small-index-subgroup},
if $\cC$ is any moiety of $\cB \setminus \cD,$
the stabilizer $S=\Gamma_{(\cB \setminus \cC),\{\str{\cC}\}}$
is contained in $\Sigma.$ Now by the definition of a BMN-variety, $S$
together with $\symf \cB_{(\cD)}$
generates the subgroup $\Gamma_{(\cD),\{\str{\cB \setminus \cD}\}} \cong
\aut{\str{\cB \setminus \cD}}.$
\end{proof}

It is interesting to compare Proposition \theProp\ with the situation
in the general case described below in Proposition \ref{SmInd=>Perf}
(a straightforward generalization of the
corresponding result from \cite{BrRom}). One of the
leading themes of the paper \cite{BrRom} by R.~Bryant and V.~Roman'kov
is the study of subgroups of index $< 2^{\aleph_0}$
of the automorphism groups of relatively free
algebras $F$ of arbitrary infinite rank. Note that the definitions
of the small index property given in \cite{BrRom} and that one
in the present paper are different: the definition in \cite{BrRom}
requires, regardless of $\rank(F),$ existence of pointwise
stabilizers of finite sets in all subgroups of $\aut F$
of index $< 2^{\aleph_0}.$ However the automorphism group $\aut F$
of a relatively free algebra $F$ of infinite
rank might simply not possess proper subgroups of index $< 2^{\aleph_0}$
as Proposition \ref{BMN=>Harami} demonstrates.

\begin{Lem} \label{i_ya_kak_sharik}
Let $F$ be a relatively free algebra with
a basis $\cB$ of infinite cardinality $\vk$ and
let $\cD$ be a subset of $\cB$ of cardinality
$< \vk.$ Then

{\rm (i)} for every automorphism $\s \in \aut F$
there is a $\cB$-moietous automorphism $\rho$ of $F$ such that
$\s|_{\cD}=\rho|_{\cD}$ \textup{(}cf. \cite[Lemma 2.1]{BrRom}\textup{)};

{\rm (ii)} the automorphism group $\Gamma=\aut F$ of $F$
is generated by all conjugate subgroups
of $\Gamma_{(\cD)}$ by elements of $\symf\cB$:
$$
\Gamma=\str{\pi \Gamma_{(\cD)} \pi\inv : \pi \in \symf\cB}
$$
\textup{(}cf. \cite[Lemma 4.3]{BrRom}\textup{)}.
\end{Lem}

\begin{proof}
We follow the ideas of the proofs of Lemma 2.1
and Lemma 4.3 of \cite{BrRom}
which correspond to (i) and (ii), respectively,
in the case when $\cD$ is finite.

(i) Let $\cE$ be a subset of $\cB$ such that
\begin{equation}
\s \cD \sle \str{\cE}.
\end{equation}
Clearly, if $\cD$ is finite, there exists a finite $\cE \sle \cB$
with (\theequation); if $\cD$ is of infinite cardinality $\nu < \vk,$
then there exists an $\cE \sle \cB$ with (\theequation) of cardinality $\nu < \vk.$

Let $\cC$ be a moiety of $\cB \setminus \cE.$ Then there is a bijection
$m : \cB \to \cC \cup \cE$ which takes every element of $\cE$
to itself, and hence an isomorphism $\mu : \str{\cB} \to \str{\cC \cup \cE}$
of algebras $F=\str\cB$ and $\str{\cC \cup \cE}$ which extends
$m;$ in particular, $\mu(e)=e$ for all $e \in \cE.$ It follows that the map
$\mu \s \mu\inv$ is an automorphism of the subalgebra
$\str{\cC \cup \cE}$ which coincides with $\s$ on $\cD$:
$$
\mu \s \mu\inv(d) = \mu \s(d) =\mu(\s d)=\s(d) \qquad [d \in \cD].
$$
Then the automorphism
$$
(\mu \s \mu\inv) \circledast \id
$$
which corresponds to the partition $\cB = (\cC \cup \cE) \sqcup (\cC \cup \cE)^c$
can be taken as a required $\cB$-moietous automorphism $\rho.$

(ii) Let $\s \in \Gamma.$ Then by (i) there is a $\cB$-moietous
automorphism $\rho$ of $F$ such that $\rho\inv \s =\gamma \in \Gamma_{(\cD)}.$
It is easy to find a $\cB$-permutational automorphism $\pi\in \symf\cB$ and
a $\cB$-moietous automorphism $\rho_0 \in \Gamma_{(\cD)}$
such that $\rho = \pi \rho_0 \pi\inv.$ Therefore,
$$
\s = \pi \rho_0 \pi\inv \cdot \gamma \in  \pi \Gamma_{(\cD)} \pi\inv \cdot \Gamma_{(\cD)}
$$
and the result follows.
\end{proof}

\begin{Prop} \label{SmInd=>Perf}
Let $F$ be a relatively free algebra of infinite
rank $\vk$ which has the small index property. Then

{\rm (i)} every proper normal subgroup of $\Gamma=\aut F$
has index $> \vk;$

{\rm (ii)} the automorphism group $\aut F$ of $F$
is perfect, that is, $[\Gamma,\Gamma]=\Gamma.$
\end{Prop}

\begin{proof} We argue like in \cite{BrRom}.

(i) If a normal subgroup $N$ of $\Gamma$ has index
$\le \vk,$ then  by part (ii) of Lemma \ref{i_ya_kak_sharik}
$N=\Gamma.$

(ii). It is known that every
nonzero abelian group has a nonzero countable quotient. This implies
that if $[\Gamma,\Gamma]$ is a proper subgroup of $\Gamma,$
then $\Gamma$ has a proper normal subgroup of countable
index which contradicts (i).
\end{proof}

The next proposition introduces a class of (necessarily
BMN-) varieties all whose free algebras of infinite rank have
the small index property. We shall show below that
any variety $\mathfrak N_c$ of all nilpotent
groups of class $\le c$ is in this class.

\begin{Prop} \label{Stab=<Shariki>}
Let $\mathfrak V$ be a variety of algebras such that
for every free algebra $F$ of infinite rank $\vk$ in $\mathfrak V$ and
every basis $\cB$ of $F$ the following is true:
given any subset $\cD$ of $\cB$ of cardinality
$< \vk,$ the stabilizer $\Gamma_{(\cD)}$ where $\Gamma=\aut F$
is generated by all $\cB$-moietous automorphisms
of $F$ which fix $\cD$ pointwise:
$$
\Gamma_{(\cD)} =\str{\mobf{}_{(\cD)}}.
$$
Then $\mathfrak V$ is a BMN-variety and
every free algebra of infinite
rank in $\mathfrak V$ has the small index property.
\end{Prop}

\begin{proof} The second statement follows from Proposition \ref{MStab-in-a-small-index-subgroup}.
To show that $\mathfrak V$ is a BMN-variety, apply part (i) of
Lemma \ref{i_ya_kak_sharik}.
\end{proof}

\section{The small index property for free nilpotent groups}

\begin{Th} \label{SmIndProp4Nilps}
Let $c \ge 1$ be a natural number. Then the variety $\mathfrak N_c$
of all nilpotent groups of class $\le c$ satisfies the conditions of Proposition
{\rm \ref{Stab=<Shariki>}}. Consequently, every infinitely generated
free nilpotent/abelian group has the small index property.
\end{Th}

\begin{proof}
Let $N$ be an infinitely generated free nilpotent group
of nilpotency class $c.$ Consider a basis $\cB$ of $N$ and let $\cD$ be a subset of $\cB$ of cardinality
$< \vk=\rank(N).$ Write $\cE$ for $\cB \setminus \cD$ and
let
$$
\cE=\{e_i : i \in I\}
$$
where $I$ is an index set of cardinality $\vk.$

Write $\Phi$ for the subgroup of $\Gamma$ generated by  $\mathrm{Mo}_\cB(N)_{(\cD)}.$ Thus our aim is to show that $\Phi$ is
equal to $\Gamma_{(\cD)}.$

We use induction on $c.$ Let then $c=1,$ which means
that $N$ is a free abelian group. Suppose that $\s \in \Gamma_{(\cD)}$ where $\Gamma=\aut N.$
Then
$$
N = \str{\cD} \oplus \str{\cE} = \str{\cD} \oplus \str{\s \cE}.
$$
It follows that there exists a basis $\{e_i' : i \in I\}$
of the group $\str\cE$ and elements
$u_i \in \str{\cD}$  such that
$$
\s e_i = e_i' + u_i \qquad [i \in I]
$$
and
$$
\{{e_i'} + {u_i} : i \in I\}
$$
is a basis of $\str{{\s\cE}}.$

As the variety of all abelian groups is a BMN-variety \cite{To_Berg},
the automorphism $\gamma \in \Gamma_{(\cD)}$
such that
$$
\gamma_0 e_i = e_i' \qquad [i \in I]
$$
is a product of $\cB$-moietous automorphisms from $\Gamma_{(\cD)},$
and hence an element of $\Phi.$ Then
$$
\gamma_0^{-1} \s e_i = e_i + u_i \qquad [i \in I].
$$

Let $I = I_1 \sqcup I_2$ be a partition of $I$ into moieties.
Clearly, both automorphisms $\pi_1,\pi_2 \in \Gamma_{(\cD)}$ where
\begin{alignat*}5
\pi_1 e_i &= e_i + u_i, &&\quad &&\pi_2 e_i=e_i,    &&[i \in I_1],\\
\pi_1 e_i &= e_i,     &&      &&\pi_2 e_i=e_i+ u_i\quad&&[i \in I_2]
\end{alignat*}
are in $\Phi$ and their product $\gamma_1=\pi_1 \pi_2$ takes
$e_i$ to $e_i+ u_i$ for all $i \in I.$ Hence
$$
\gamma_1^{-1} \gamma_0^{-1} \s e_i = e_i \qquad [i \in I],
$$
and $\s$ is in $\Phi.$

Recall that $\gamma_k(G)$ where $k \ge 1$ denotes the $k$-th
term of the lower central series of a group $G$ (defined
inductively as follows: $\gamma_1(G)=G$ and
$\gamma_{k+1}(G)=[\gamma_k(G),G]).$

We consider the induction step. Let $N$
be a free nilpotent group of nilpotency
class $c \ge 2.$ Then $\gamma_{c+1}(N)=\{1\},$
$\gamma_{c}(N)$ is the center of $N,$
and the quotient group $M=N/\gamma_c(N)$
is a free nilpotent group of nilpotency class
$c-1.$ The main result of \cite{BrMa} implies
that the natural homomorphism $\aut N \to \aut M$
determined by the natural homomorphism
$\eps: N \to M$ is surjective.

We claim that any $\eps \cB$-moietous
automorphism $s \in \aut{M}_{(\eps \cD)}$ of $M$ can be lifted to a $\cB$-moietous
automorphism $\s \in \aut N_{(\cD)}=\Gamma_{(\cD)}$ of $N.$ Indeed, assume that
$s$ preserves the subgroup $\str{\eps \cC}$
and fixes pointwise $\eps \cB \setminus \eps \cC$
where $\cC$ is a moiety of $\cB$ containing $\cD.$ As the natural
homomorphism $\aut{\str \cC} \to \aut{\str{\eps \cC}}$
of the automorphism groups of infinitely generated
free nilpotent groups $\str \cC$ and $\str{\eps \cC}$
is surjective by the quoted result of \cite{BrMa},
there is a $\cB$-moietous automorphism $\s_0$ of $N$
which induces $s$ such that: (a) $\s_0$ fixes pointwise $\cB \setminus \cC;$
(b) $\s_0$ preserves the subgroup $\str{\cC}$ and
(c) $\s_0$ fixes all elements of $\cD$ modulo
$\gamma_c(N).$ It is well-known that given any family
$\{c_b : b \in \cB\}$ of elements of $\gamma_2(N),$
there is an IA-automorphism of $N$ which takes
$b$ to $b c_b$ $(b \in \cB).$ Let then $\s_1$
be the automorphism of $N$ which coincides with
$\s_0$ on $\cD$ and fixes $\cB \setminus \cD$
pointwise; clearly, $\s_1$ preserves $\str{\cC}$
and fixes $\cB \setminus \cC$ pointwise. Now $\s=\s_1^{-1} \s_0 \in \Gamma_{(\cD)}$
is a $\cB$-moietous
automorphism of $N$ which induces $s,$ as desired.

The induction hypothesis then implies that for every $\s \in \Gamma_{(\cD)}$
there is an element $\gamma \in \Phi$
such that $\a=\gamma\inv \s \in \Gamma_{(\cD)}$ induces the
 trivial automorphism of $M=N/\gamma_c(N).$

We claim that $\a$ belongs to $\Phi.$ Suppose that $\a$
acts on $\cE=\{e_i : i \in I\}$ as follows:
\begin{equation} \label{alpha_from_I_c}
\a e_i =e_i w_i \qquad [i \in I]
\end{equation}
where $w_i=w(e_i) \in \gamma_c(N).$

Observe that every automorphism $\delta$ of $N$ which
fixes $N$ pointwise modulo the subgroup $\gamma_c(N),$
\begin{equation}
\delta x \equiv x \Mod{\gamma_c(N)},
\end{equation}
takes every element of $\gamma_c(N)$ to itself and that
any two automorphisms of $N$ of the form (\theequation) are then
commuting.

Let us partition $\cE = \cB \setminus \cD$ into moieties:
$$
\cB \setminus \cD = \cF \sqcup \cG
$$
and define automorphisms $\a_{\cF}, \a_\cG \in \Gamma_{(\cD)}$ of $N$ as follows:
\begin{alignat*}  5
\a_\cF f &=f,          && \a_\cG f &&= \a f,\quad && [f \in \cF],\\
\a_\cF g &=\a g,\quad &&  \a_\cG g &&= g,         &&  [g \in \cG].
\end{alignat*}
Then $\a = \a_\cF \a_\cG = \a_\cG \a_\cF$ and $\a_\cF \in \Gamma_{(\cD \cup \cF)},$
$\a_\cG \in \Gamma_{(\cD \cup \cG)}.$

Let us prove that $\a_\cF$ is in $\Phi.$ By symmetry, it will imply
that $\a_\cG$ is also in $\Phi,$ whence the result.

Partition $\cF$ into $(c+1)$ moieties:
$$
\cF = \cF_1 \sqcup \ldots \sqcup \cF_{c+1}.
$$
Every element of $\gamma_c(N)$ can be written
as a product of left-normed basic commutators
$[b_1,\ldots,b_c]$ of weight $c$ where $b_1,\ldots,b_c$ run over $\cB.$ For every $i \in I,$
write then $w_i=w(e_i) \in \gamma_c(N),$ participating in (\theequation) above, as
$$
w_i=w(e_i) =t_{1}(e_i) \ldots t_{c+1}(e_i)
$$
where $t_{k}(e_i)\in \gamma_c(N)$ is a product of left-normed basic commutators
of weight $c$ which
do {\it not} have occurrences from $\cF_k$ ($k=1,\ldots,c+1).$
For every $k=1,\ldots,c+1,$ define $\beta_k \in \Gamma_{(\cD \cup \cF)}$
as follows:
$$
\beta_k g = g t_k(g) \qquad [g \in \cG].
$$
As
$$
t_k(g) \in \str{\cD \cup (\cF \setminus \cF_k) \cup \cG}=\str{\cB \setminus \cF_k}
$$
for all $g \in \cG,$ we have that $\beta_k$ fixes
$\cF_k$ pointwise and preserves the subgroup $\str{\cB\setminus\cF_k}$
setwise $(k=1,\ldots,c+1).$ It follows that $\beta_k$
is a $\cB$-moietous automorphism of $N$ which fixes
$\cD$ pointwise, or, in other words, an element of $\Phi.$
Finally, as
$$
\a_\cF = \beta_1 \ldots \beta_{c+1},
$$
so is $\a_\cF.$
\end{proof}

\begin{Cor}
Let $F$ be a free group of infinite rank and let $V(F)$
be a verbal subgroup of $F$ which contains a term
$\gamma_c(F)$ where $c \ge 2$ of the lower central series of $F.$
Then the group $F/V(F)$ has the small index property.
\end{Cor}

\begin{proof}
By the main result of \cite{BrMa} we have already referred to
above, the natural homomorphisms
$$
\aut F \to \aut{F/\gamma_c(F)} \text{ and } \aut F \to \aut{F/V(F)}
$$
determined by the natural homomorphisms
$$
F \to F/\gamma_c(F) \text{ and } F \to F/V(F),
$$
respectively, are surjective. It follows that the natural
homomorphism
$$
\eps : \aut{F/\gamma_c(F)} \to \aut{F/V(F)}
$$
determined by the natural homomorphism $F/\gamma_c(F) \to F/V(F)$ is also surjective.

Let $\Sigma \le \aut{F/V(F)}$ be a subgroup
of index at most $\rank(F/V(F))=\rank(F).$ Then the full
preimage $\eps\inv(\Sigma)$ of $\Sigma$ is of index
at most $\rank(F/\gamma_c(F))=\rank(F)$ in $\Gamma=\aut{F/\gamma_c(F)}.$
It follows that there is a subset $\cD$ of $F/\gamma_c(F)$
of cardinality $< \rank(F)$ such that
$$
\Gamma_{(\cD)} \le \eps\inv(\Sigma),
$$
whence
$$
\eps(\Gamma_{(\cD)}) \le \Sigma,
$$
and the former subgroup is the pointwise stabilizer
of a subset of $F/V(F)$ of cardinality $< \rank(F)$ of $F/V(F).$
\end{proof}

The problem whether a free group $F$ of arbitrary infinite
rank has the small index property seems to be very intriguing
(as we remarked in the Introduction, the answer is affirmative
in the case when $F$ is countable \cite{BrEv}). The following partial
answer is provided by Theorem \ref{SmIndProp4Nilps}.

\begin{Cor}
Let $F$ be a free group of infinite rank $\vk.$
Suppose that $\Sigma$ is a subgroup of $\Gamma=\aut F$ of index at most $\vk$
such that
$$
\Sigma= \bigcap_{c \ge 2} \Sigma I_c
$$
where $I_c$ is the kernel of the natural homomorphism
$\aut F \to \aut{F/\gamma_c(F)}$ determined by
the natural homomorphism $F \to F/\gamma_c(F)$
$(c \ge 2).$ Then $\Sigma$ contains the pointwise
stabilizer $\Gamma_{(\cD)}$ of a subset $\cD$
of $F$ of cardinality $< \vk.$
\end{Cor}

\begin{proof}
Let a natural number $c \ge 2$ be arbitrary. Write $\eps$
for the natural homomorphism $F \to N_c=F/\gamma_c(F).$ It is
convenient to use the same symbol $\eps$ for the
natural homomorphism $\aut F \to \Gamma^{(c)}=\aut{F/\gamma_c(F)}.$
Take a basis $\cB$ of $F.$ As $\Sigma$ is a subgroup of index at most $\vk,$
by Proposition \ref{MStab-in-a-small-index-subgroup}
there is a subset $\cD$ of $\cB$ of cardinality $< \vk$ such that
$$
\mathrm{Mo}_\cB(F)_{(\cD)} \sle \Sigma,
$$
and hence
$$
\eps(\mathrm{Mo}_\cB(F)_{(\cD)}) = \mathrm{Mo}_{\eps \cB}(N_c)_{(\eps \cD)} \le \eps(\Sigma).
$$
Now $\eps \cB$ is a basis of $N_c,$ and then by Theorem \ref{SmIndProp4Nilps},
$$
\Gamma^{(c)}_{(\eps \cD)}=\str{\mathrm{Mo}_{\eps \cB}(N_c)_{(\eps \cD)}} \le \eps(\Sigma).
$$
Lifting all objects that participate in the last inequality back
to $\Gamma,$ we get that
$$
\Gamma_{(\cD)} I_c \le \Sigma I_c \To \Gamma_{(\cD)} \le \Sigma I_c
$$
for all $c \ge 2.$ It follows that
$$
\Gamma_{(\cD)} \le \bigcap_{c \ge 2} \Sigma I_c = \Sigma,
$$
as claimed.
\end{proof}

\begin{rem}
\rm Subgroups $\Sigma$ of $\Gamma=\aut F$ that satisfy the condition
$\Sigma=\bigcap_{c \ge 2} \Sigma I_c$ can be characterized
as {\it closed} subgroups of $\Gamma$ with respect to the topology $\tau$
on $\Gamma$ determined by the filtration $(I_c : c \ge 2).$ Thus
having an arbitrary subgroup $\Sigma$ of $\Gamma$ of index
at most $\vk$ we could only claim that the closure $\avst \Sigma$
of $\Sigma$ in $\tau$ contains the pointwise stabilizer of a subset of
$F$ of cardinality $< \vk.$
\end{rem}

Using the fact that any variety of vector spaces over a fixed division
ring is a BMN-variety \cite{To_Berg}, we can easily adapt the `abelian' part of the proof
of Theorem \theTh\ to prove the following result.

\begin{Prop}
Let $V$ be an infinite-dimensional vector space
over a division ring. Then $V$ has the small
index property.
\end{Prop}

A similar result can be in fact proven for
free modules of infinite rank over rings from
a rather large class of rings the author considered in \cite{To_Berg}
(see Remark 2.3 and Theorem 2.4 of \cite{To_Berg}
for the details).

Theorem \ref{SmIndProp4Nilps} can be also used to show
that all automorphisms of the automorphism group
$\aut A$ of an infinitely generated free abelian group $A$
are inner. The crucial step in the proof is the
reconstruction of the family of all unimodular
elements of $A$ in $\aut A$ with the use of
the small index property.

\subsection*{Acknowledgements}
This work has been partially supported
by the Russian Federal Target Program ``Scientific, Academic and Teaching Staff of Innovative Russia''.


\begin{thebibliography}{99}

\bibitem{BrEv}
Bryant, R. \and Evans, D.
The small index property for free groups and relatively free groups.
J. London Math. Soc. (2)
55:363--369.

\bibitem{BrMa}
Bryant, R. \and Macedonska, O.
Automorphisms of relatively free nilpotent groups of infinite rank.
121:388--398.

\bibitem{BrRom}
Bryant, R. \and Roman'kov, V.
The automorphism groups of relatively free algebras.
J.  Algebra
209:713--723.

\bibitem{DiNeuTho}
Dixon, J., Neumann, P.M. \and Thomas, S.
Subgroups of small index in infinite symmetric groups.
Bull. London Math. Soc.
18:580--586.


\bibitem{Hodges}
Hodges, W.
(1993).
Model Theory.
University Press, Cambridge.

\bibitem{Kech}
Kechris, A.
(1995).
Classical descriptive set theory.
Springer-Verlag, New York.

\bibitem{Macph}
Macpherson, H.D.
Maximal subgroups of infinite-dimensional linear groups.
J. Austral. Math. Soc. (Series A)
53:338--351.

\bibitem{MN}
Macpherson, H.D. \and Neumann, P.M.
Subgroups of infinite symmetric groups.
J. London Math.  Soc. (2)
42:64--84.

\bibitem{O'M}
O'Meara, O.
A general isomorphism theory for linear groups.
J. Algebra
44:93--142.


\bibitem{McK}
McKenzie, R.
On elementary types of symmetric groups.
Algebra Universalis
1:13--20.

\bibitem{MKS}
Magnus, W., Karrass A., \and Solitar D.
(1966).
Combinatorial Group Theory.
Wiley.


\bibitem{To_Ab}
Tolstykh, V.
What does the automorphism group of a free abelian group $A$ know about $A?$
(2005).
In: Logic and its applications,
Contemp. Math. 380.
Amer. Math. Soc., Providence, RI,
pp. 283--296,

\bibitem{To_Berg}
Tolstykh, V.
On the Bergman property for the automorphism groups of relatively free groups.
J. London Math. Soc. (2)
73:669--680.

\bibitem{To_Smallness}
Tolstykh, V.
Small conjugacy classes in the automorphism groups of relatively free groups.
J. Pure Appl. Algebra
215:2086--2098.


\end{thebibliography}
\end{document}